\documentclass[11pt]{amsart}

\usepackage[pdftex]{graphicx}
\usepackage[a4paper,margin=2.5cm]{geometry}

\usepackage{amsfonts}
\usepackage{amsmath}
\usepackage{amssymb}
\usepackage{amsthm}
\usepackage{dsfont}
\usepackage{tikz}
\usepackage{bm}
\usepackage{thmtools}

\usepackage{url}

\usepackage{rotating}

\usepackage[T1]{fontenc}

\usepackage{paralist}
\usepackage{color}

\usepackage{overpic}

\usepackage[colorlinks,cite color=blue,pagebackref=true,pdftex]{hyperref}
\usepackage{cleveref}
\usepackage{float}

\renewcommand\emptyset{\varnothing}

\newcommand\ord{\mathcal{O}}
\newcommand\ordC{\mathcal{L}}

\renewcommand\l{\lambda}
\newcommand\Z{\mathbb{Z}}

\newcommand\R{\mathbb{R}}

\newcommand\OGT{\mathcal{GT}}

\renewcommand\int{\mathsf{int}}

\newcommand{\PS}{\operatorname{PS}}

\newtheorem{thm}{Theorem}[section]

\newtheorem{prop}[thm]{Proposition}

\theoremstyle{definition}

\newtheorem{example}[thm]{Example}

\title[Ehrhart positivity for marked order polytopes]{Ehrhart positivity for marked order polytopes}

\author{Katharina Jochemko}
\author{Krishna Menon}
\address{KTH Royal Institute of Technology, %
Department of Mathematics, %
SE-100 44 Stockholm, Sweden}
\email{\{jochemko,puzhan\}@kth.se}


\date{\today}
\thanks{KJ was partially supported by the Wallenberg AI, Autonomous Systems and Software program funded by the Knut and Alice Wallenberg foundation, grant nr 2023-04063 from the Swedish research council and the Verg
Foundation. KM was partially supported by the Verg foundation and the Göran Gustafsson foundation.}

\begin{document}

\begin{abstract}
    Given a pair of finite posets $A \subseteq P$, the function counting integer-valued order preserving extensions of an order preserving map $\l : A
    \rightarrow \mathbb{Z}$ from $A$ to $P$ is given by a piecewise polynomial in $\l$. We provide a criterion for the nonnegativity of the coefficients of these multivariate polynomials and apply it to show that marked order polytopes of skew shapes are Ehrhart positive in a multivariate sense. This extends recent results of Ferroni-Morales-Panova on order polytopes of skew shapes and proves conjectures on the Ehrhart positivity of skew Gelfand-Tsetlin polytopes and $m$-generalized  Pitman-Stanley polytopes due to Alexanderson-Alhajjar and Dugan-Hegarty-Morales-Raymond, respectively. 
\end{abstract}

\maketitle

\section{Introduction and Results}
The enumeration of order preserving maps is a classical topic in enumerative combinatorics. Given a finite partially ordered set (poset) $P$, a map $\l$ from  $P$ into the $n$-chain $[n]=\{1,\ldots, n\}$ is called order preserving if $\l (p)\leq \l(q)$ whenever $p\prec _P q$. A famous result by Stanley~\cite{Stanleychromatic-like} states that the cardinality of these maps is given by a polynomial $\Omega _P(n)$ for positive integers $n$, called the order polynomial of the poset $P$. 

The polynomiality of $\Omega _P(n)$ can be obtained via polyhedral-geometric methods and constitutes a prime example of an application of Ehrhart theory to enumeration questions: a fundamental theorem by Ehrhart~\cite{Ehrhart} states that for any lattice polytope $Q\subset \R ^d$, the number of lattice points in the $n$-th dilate of $Q$, $|nQ\cap \mathbb{Z}^d|$, is given by a polynomial $E_Q (n)$ for $n\geq 0$, called the Ehrhart polynomial of $Q$. By identifying order preserving maps into the $n$-chain with lattice points in dilates of the order polytope
\[
\mathcal{O}_P \ = \ \{\l \colon P \rightarrow [0,1] \text{ order preserving}\} \subset \mathbb{R}^P \, ,
\]
Stanley~\cite{Stanleytwoposetpolytopes} obtained that $\Omega _P (n)=E_{\mathcal{O}_P}(n-1)$, thereby providing a further, geometric proof of the polynomiality.

A fundamental question is to characterize Ehrhart polynomials and to interpret their coefficients. Ehrhart polynomials can have negative coefficients in general and there exist also examples for order polytopes~\cite{StanleyMathoverflow,LiuTsuchiya}. Identifying geometric and combinatorial properties that lead to nonnegative coefficients in Ehrhart polynomials is thus of particular interest; see the surveys~\cite{LiuSurvey,ferronihigashitani}. Lattice polytopes whose Ehrhart polynomials have only nonnegative coefficients are called Ehrhart positive.

In recent work, Ferroni-Morales-Panova~\cite{skewshapes} showed nonnegativity of the coefficients of the order polynomial for skew-shaped posets, which include zig-zag and more general fence posets. 
The poset $P$ associated to a skew-shape $\l/\mu$ for partitions $\lambda=(\lambda _1\geq \lambda _2 \geq \cdots )$ and $\mu=(\mu _1\geq \mu _2 \geq \cdots )$ with $\mu _i\leq \lambda _i$ is constructed as follows: 
The elements of the poset correspond to the cells of the skew-shape, i.e., $\{(i, j) \colon j \in [\mu_i + 1, \l_i]\}$. 
The covering relations are given by $(i + 1, j) \prec_P (i, j)$ and $(i, j + 1) \prec_P (i, j)$ (see \Cref{fig:skew}).

\begin{figure}
    \centering
    \begin{tikzpicture}[scale = 0.6]
        \draw (0, 0) grid (3, 1);
        \draw (1, 1) grid (3, 2);
        \draw (1, 2) grid (5, 3);
        \draw (2, 3) grid (6, 4);
        \node at (7, 2) {$\rightarrow$};
        \node at (3, -1) {\(\l/\mu\)};
    \end{tikzpicture}
    \hspace{0.2cm}
    \begin{tikzpicture}[scale = 0.6]
        \draw (0, 0) grid (3, 1);
        \draw (1, 1) grid (3, 2);
        \draw (1, 2) grid (5, 3);
        \draw (2, 3) grid (6, 4);
        \draw (1.5, 0.5) -- (1.5, 2.5);
        \draw (2.5, 0.5) -- (2.5, 3.5);
        \draw (3.5, 2.5) -- (3.5, 3.5);
        \draw (4.5, 2.5) -- (4.5, 3.5);

        \draw (0.5, 0.5) -- (2.5, 0.5);
        \draw (1.5, 1.5) -- (2.5, 1.5);
        \draw (1.5, 2.5) -- (4.5, 2.5);
        \draw (2.5, 3.5) -- (5.5, 3.5);

        \node[circle, fill = black, inner sep = 1pt] at (0.5, 0.5) {};
        \node[circle, fill = black, inner sep = 1pt] at (1.5, 0.5) {};
        \node[circle, fill = black, inner sep = 1pt] at (1.5, 1.5) {};
        \node[circle, fill = black, inner sep = 1pt] at (1.5, 2.5) {};
        \node[circle, fill = black, inner sep = 1pt] at (2.5, 0.5) {};
        \node[circle, fill = black, inner sep = 1pt] at (2.5, 1.5) {};
        \node[circle, fill = black, inner sep = 1pt] at (2.5, 2.5) {};
        \node[circle, fill = black, inner sep = 1pt] at (2.5, 3.5) {};
        \node[circle, fill = black, inner sep = 1pt] at (3.5, 2.5) {};
        \node[circle, fill = black, inner sep = 1pt] at (3.5, 3.5) {};
        \node[circle, fill = black, inner sep = 1pt] at (4.5, 2.5) {};
        \node[circle, fill = black, inner sep = 1pt] at (4.5, 3.5) {};
        \node[circle, fill = black, inner sep = 1pt] at (5.5, 3.5) {};
        \node at (7, 2) {$\rightarrow$};
        \node at (3, -1) {};
    \end{tikzpicture}
    \hspace{0.2cm}
    \begin{tikzpicture}[scale = 0.7, rotate = -45]
            \draw (1.5, 0.5) -- (1.5, 2.5);
        \draw (2.5, 0.5) -- (2.5, 3.5);
        \draw (3.5, 2.5) -- (3.5, 3.5);
        \draw (4.5, 2.5) -- (4.5, 3.5);

        \draw (0.5, 0.5) -- (2.5, 0.5);
        \draw (1.5, 1.5) -- (2.5, 1.5);
        \draw (1.5, 2.5) -- (4.5, 2.5);
        \draw (2.5, 3.5) -- (5.5, 3.5);

        \node[circle, fill = black, inner sep = 2pt] at (0.5, 0.5) {};
        \node[circle, fill = black, inner sep = 2pt] at (1.5, 0.5) {};
        \node[circle, fill = black, inner sep = 2pt] at (1.5, 1.5) {};
        \node[circle, fill = black, inner sep = 2pt] at (1.5, 2.5) {};
        \node[circle, fill = black, inner sep = 2pt] at (2.5, 0.5) {};
        \node[circle, fill = black, inner sep = 2pt] at (2.5, 1.5) {};
        \node[circle, fill = black, inner sep = 2pt] at (2.5, 2.5) {};
        \node[circle, fill = black, inner sep = 2pt] at (2.5, 3.5) {};
        \node[circle, fill = black, inner sep = 2pt] at (3.5, 2.5) {};
        \node[circle, fill = black, inner sep = 2pt] at (3.5, 3.5) {};
        \node[circle, fill = black, inner sep = 2pt] at (4.5, 2.5) {};
        \node[circle, fill = black, inner sep = 2pt] at (4.5, 3.5) {};
        \node[circle, fill = black, inner sep = 2pt] at (5.5, 3.5) {};
        \node at (4.25, 0.75) {\(P\)};
    \end{tikzpicture}
    \caption{The skew-shape $6533/211$ and its associated poset.}
    \label{fig:skew}
\end{figure}
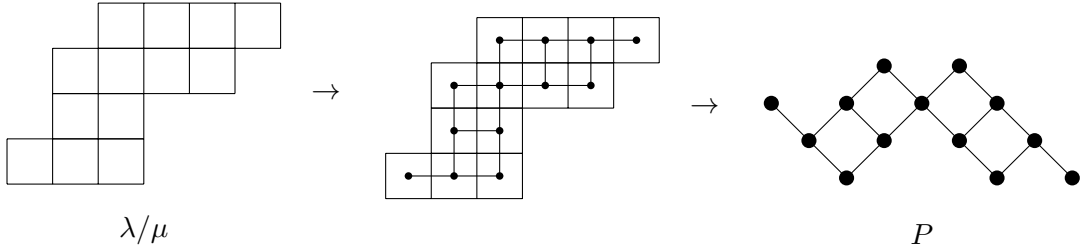

Their result implies Ehrhart positivity of the order polytope of such posets. A key ingredient of their proof is the following reduction to the nonnegativity of the linear term for families of posets that are closed under taking filters and ideals, together with a combinatorial formula for the linear term for skew-shaped posets~\cite[Proposition 4.1]{skewshapes}.

\begin{thm}[{\cite[Theorem 3.4]{skewshapes}}]\label{thm:orderpolynomialslinearterm}
    Let $\mathcal{F}$ be a family of posets closed under taking ideals and filters, and such that for every $P\in \mathcal{F}$ the linear term of $\Omega _P(n)$ is nonnegative. Then $\Omega _P(n)$ has nonnegative coefficients for any $P\in \mathcal{F}$.
\end{thm}

In this note, we apply these results and obtain a multivariate version for counting integer-valued extensions of order preserving maps. Those correspond to lattice points in marked order polytopes. Given a pair of finite posets $A\subseteq P$ and an order preserving map $\l \colon A \rightarrow \mathbb{Z}$, the marked order polyhedron
\[
\mathcal{O}_{P,A}(\l) \ = \ \{\hat{\l} \colon P\rightarrow \mathbb{R} \text{ order preserving and } \hat{\l}(a)=\l(a) \text{ for all }a \in A\} \subset \mathbb{R}^P
\]
is the set of all real-valued order preserving extensions of $\l$ from $A$ to $P$. If $A$ includes the minima and maxima of $P$, $\min (P)\cup \max (P)\subseteq A$, then $\mathcal{O}_{P,A}(\l)$ is a bounded lattice polytope and called a marked order polytope. Marked order polytopes were introduced by Ardila-Bliem-Salazar~\cite{ArdilaBliemSalazar} and contain many interesting classes of polytopes arising in combinatorics and representation theory. Examples include the Pitman-Stanley polytopes~\cite{PitmanStanley} and Gelfand-Tsetlin~\cite{GelfandTsetlinPolytopes} polytopes.

In~\cite{JochemkoSanyal} the first author together with Sanyal studied enumerative and geometric properties of marked order polytopes. The geometry of the marked order polytope $\mathcal{O}_{P,A}(\l)$ as well as the function enumerating lattice points in this polytope is governed by the order on $A$ induced by the values attained by $\l$. The order cone of the induced poset on $A\subseteq P$ is the cone $\mathcal{L}(A)=\{\l \colon A\rightarrow \mathbb{R}\colon \text{ order preserving}\}$ consisting of all order preserving maps on $A$. In particular, $\mathcal{O}_{P,A}(\l)$ is the empty polytope for all $\l$ outside of $\mathcal{L}(A)$. For order preserving $\l$ the following structural result on the counting function of order preserving maps was proved in~\cite{JochemkoSanyal}.

\begin{thm}[{\cite[Theorem 2.6]{JochemkoSanyal}}]\label{thm:piecew_poly}
    Let $A \subseteq P$ be a pair of posets with $\min(P) \cup \max(P)
    \subseteq A$. For integral-valued order preserving maps $\l : A \rightarrow \Z$, the function
    \[
        \Omega_{P,A}(\l) \ = \ |\ord_{P,A}(\l) \cap \Z^P|
    \]
    is a piecewise polynomial over the order cone $\ordC(A)$. 
\end{thm}
More precisely, for any natural labeling of the elements in $A=\{a_0,\ldots, a_r\}$ , i.e., $i<j$ whenever $a_i\prec _P a_j$, $\Omega_{P,A}(\l)$ agrees with a polynomial on $\{\l \in \mathbb{R}^A \colon \l(a_0) \leq \l (a_1)\leq \l (a_2)\leq \cdots \leq \l (a_r)\}$. In this case, $\Omega_{P,A}(\l)$ is a weighted Minkowski sum of order polytopes and the polynomial can be expressed in the variables $t_i=\l (a_i)-\l(a_{i-1})$, $i=1,\ldots r$~\cite{PitmanStanley,StanleyCombAppl,JochemkoSanyal}. For the family of marked order polytopes it is thus natural to study positivity questions for the coefficients of these multivariate counting functions.

We have the following multivariate Ehrhart positivity criterion for marked order polytopes, extending Theorem~\ref{thm:orderpolynomialslinearterm}.

\begin{thm}\label{thm:posetfamilies}
Let $\mathcal{F}$ be a family of posets closed under taking ideals and filters, such that for every $P\in \mathcal{F}$ the linear term of $\Omega _P (n)$ is nonnegative. Then for any subposet $A\subseteq P$, $\min P\cup \max P \subseteq A$, and any natural labeling $A=\{a_0,\ldots, a_r\}$, there is a polynomial $f\in \mathbb{Q}[t_1,\ldots, t_r]$ with nonnegative coefficients, such that
\[
\Omega_{P,A}(\l) \ = \ f(t_1,\ldots, t_r)
\]
for all $\l$ such that $\l (a_0)\leq \l (a_1)\leq \cdots \leq \l (a_r)$ where $t_i=\l (a_i)-\l(a_{i-1})$, $i=1,\ldots, r$.
\end{thm}
Since the $n$-th dilate of $\mathcal{O}_{P,A}(\l)$ equals $n\mathcal{O}_{P,A}(\l)=\mathcal{O}_{P,A}(n\l)$, Theorem~\ref{thm:posetfamilies} also immediately implies Ehrhart positivity in the usual, univariate sense.
\begin{thm}\label{thm:posetfamiliesunivariate}
Let $\mathcal{F}$ be a family of posets closed under taking ideals and filters, such that for every $P\in \mathcal{F}$ the linear term of $\Omega _P (n)$ is nonnegative. Then for any subposet $A\subseteq P$, $\min P\cup \max P \subseteq A$, and any $\l \colon A\rightarrow \mathbb{Z}$, $\mathcal{O}_{P,A}(\l)$ is Ehrhart positive.
\end{thm}

Theorem~\ref{thm:posetfamiliesunivariate} in turn immediately implies that marked order polytopes on skew shapes are Ehrhart positive, using the result by Ferroni-Morales-Panova~\cite[Proposition 4.1]{skewshapes} that the linear term of order polynomials of skew shapes is nonnegative.
\begin{thm}
    Let $P$ be a skew-shape and $A\subseteq P$ be a subset containing the minima and maxima of $P$, and $\l \colon A\rightarrow \mathbb{Z}$. Then the marked order polytope $\Omega_{P,A}(\l)$ is Ehrhart positive.
\end{thm}
Examples that fall under this umbrella are the Pitman-Stanley polytopes, for which the Ehrhart positivity was already shown in~\cite{PitmanStanley}. Our result yields a multivariate extension of that positivity result to $m$-generalized Pitman-Stanley polytopes~\cite{mgeneralized,PitmanStanley} which enumerate plane partitions (Theorem~\ref{thm:conjmgeneralized}). Thereby we answer \cite[Conjecture 3.4]{generalizedPitmanStanleyAbstract} due to Dugan-Hegarty-Morales-Raymond in the affirmative. (See also \cite[Conjecture 9.3]{skewshapes}.) 

Another important class of marked order polytopes are Gelfand-Tsetlin polytopes~\cite{GelfandTsetlinPolytopes}. Ehrhart positivity of the (straight, non-skew-shaped) Gelfand-Tsetlin polytope follows from the hook-content formula \cite[Equation 7.106]{ec2}. We show that multivariate Ehrhart positivity also holds for Gelfand-Tsetlin polytopes $\mathcal{GT}_{\lambda/ \mu}$ of skew-shapes $\lambda/ \mu$ (Theorem~\ref{thm:conjGelfand-Tsetlin}), thereby proving a conjecture of Alexandersson-Alhajjar~\cite[Conjecture 7]{GelfandTsetlinconjecture}. Ideals and filters of the underlying poset in this case are shifted skew shapes for which Ehrhart positivity is conjectured~\cite{skewshapes} (but not proven to the best of our knowledge). Theorem~\ref{thm:posetfamilies} can thus not be applied directly. However, as we argue, slightly weaker assumptions than in Theorem~\ref{thm:posetfamilies} are sufficient to obtain Ehrhart positivity in this particular case.

We conclude by remarking that the statements on multivariate Ehrhart positivity for marked order polytopes presented in this note carry over verbatim to marked chain polytopes~\cite{ArdilaBliemSalazar} and more general marked chain-order polytopes~\cite{markedchainorder} on the same underlying poset, as their lattice points are in bijection.

\section{Proof of Theorem~\ref{thm:posetfamilies}}

The key ingredient of our argument is the following expression of $\Omega _{P,A}(\l)$ in terms of order preserving maps of subposets, in the spirit of~\cite[Theorem 2.6]{SWZ}. Recall that a subset $I\subseteq P$ is called an ideal if $I$ is downward closed, i.e., $p\in I$ implies $q\in I$ for all $q\prec _P p$. Similarly, a subset $I\subseteq P$ is called a filter if it is upward closed, equivalently, if $P\setminus I$ is an ideal. For all $p\in P$, let $\uparrow p$ denote the filter $\{q\in P\colon p\preceq _P q\}$.
\begin{prop}\label{prop:productformula}
Let $A\subseteq P$ be a pair of posets with $\min P \cup \max P\subseteq A$ and  natural labeling $A=\{a_0,\ldots, a_r\}$. Then for $\l (a_i) =t_0+t_1+\cdots + t_i$, $i=0,\ldots, r$, we have the identity of polynomials
\[
\Omega _{P,A}(\l) \ = \ \sum _{I_0 \subset I_1 \subset \cdots \subset I_r\atop a_i \in I_i} \prod _{i=1}^r \Omega _{I_i\setminus (I_{i-1}\cup \uparrow a_i)}(t_i)
\]
where the sum is over all strict chains $I_0 \subset I_1 \subset \cdots \subset I_r$ of order ideals of $P$ such that $a_i\in I_i\setminus I_{i-1}$ for all $i=0,\ldots, r$ and $I_{-1}:=\emptyset$.
\end{prop}
\begin{proof}
    As both sides are polynomials, it is enough to prove the identity for positive integers $t_1, \ldots, t_r$. For that, we construct a bijection between the integer points in the marked order polytope $\mathcal{O}_{P,A}(\l)$ and $\bigcup _I \prod _{i=1}^r (t_i-1) \mathcal{O}_{I_i\setminus (I_{i-1}\cup \uparrow a_i)}$ where the union is over all chains of order ideals as in the assumption: any order preserving extension $\hat{\l}\in \mathcal{O}_{P,A}(\l)\cap \mathbb{Z}^P$ of $\l$ from $A$ to $P$ gives rise to a strict chain of order ideals by setting $I_i=\{p\in P\colon \hat{\l}(p)\leq t_0+t_1+\cdots +t_i=\lambda (a_i)\}$. Since the $t_i$s are positive we have $a_i\in I_i\setminus I_{i-1}$. For $i\geq 1$ and $p\in I_i\setminus I_{i-1}$, let $g(p)=\hat{\l}(p) -(t_0+t_1+\cdots +t_{i-1}+1)$. Then $0\leq g(p)\leq t_i-1$, $g(p)$ is order preserving on the subposet $I_i\setminus I_{i-1}$ and furthermore we have $g(p)=t_i-1$ for all $p\in I_i \mathop{\cap} \uparrow a_i$. Thus $g$ restricted to ${I_i\setminus (I_{i-1}\cup \uparrow a_i)}$ is an order preserving map with integer values in $[0,t_i-1]$ and $(g_{I_i\setminus (I_{i-1}\cup \uparrow a_i)})_{i=1,\ldots, r}\in \bigcup _I \prod _{i=1}^r (t_i-1) \mathcal{O}_{I_i\setminus (I_{i-1}\cup \uparrow a_i)}$ is a well-defined map from the integer points of $\mathcal{O}_{P,A}(\l)$ to integer points in $\bigcup _I \prod _{i=1}^r (t_i-1) \mathcal{O}_{I_i\setminus (I_{i-1}\cup \uparrow a_i)}$. This map is bijective and the inverse can be defined as follows: if $p\in I_0$ we set $\hat{\l}(p)=t_0$. If $p\in I_i\setminus (I_{i-1}\cup \uparrow a_i)$ then let $\hat{\l}(p)=g_{I_i\setminus (I_{i-1}\cup \uparrow a_i)}(p)+t_0+\cdots + t_{i-1}+1$. And if $p\in I_i \mathop{\cap} \uparrow a_i$ then set $\hat{\l}(p)=t_0+t_1+\cdots + t_i$.
\end{proof}

\begin{proof}[Proof of Theorem~\ref{thm:posetfamilies}]
    By Proposition~\ref{prop:productformula} we can write $\Omega _{P,A}(\l)$ as a sum of products of order polynomials of posets obtained from taking ideals and filters. These posets are by assumption also in the class $\mathcal{F}$. By the assumption on the linear term and Theorem~\ref{thm:orderpolynomialslinearterm}, all these order polynomials in the product have nonnegative coefficients and thus also $\Omega _{P,A}(\l)$ has nonnegative coefficients as a polynomial in $t_i$.
\end{proof}

\section{Applications}
\subsection{Generalized Pitman-Stanley polytopes}
Let $\bm y, \bm z \in \Z_{\geq 0}^k$ and $m \geq 1$. 
For any $i \in [k]$, set $\tilde y_i = y_1 + \cdots + y_i$ and similarly define $\tilde z_i$. 
The \emph{$m$-generalized Pitman–Stanley polytope}, denoted $\PS_k^{m}(\bm y, \bm z)$, is the set of all points $(x_{i, j}) \in \R^{km}$ that satisfy
\begin{itemize}
    \item $\tilde z_i \leq x_{i, 1} \leq x_{i, 2} \leq \cdots \leq x_{i, m} \leq \tilde y_i$ for all $i \in [k]$, and
    \item $x_{i, j} \leq x_{i + 1, j}$ for all $i \in [k - 1]$ and $j \in [m]$.
\end{itemize}
The lattice points of this polytope are in correspondence with plane partitions of the skew-shape $(\tilde y_k, \ldots, \tilde y_2, \tilde y_1)/ (\tilde z_k, \ldots, \tilde z_2, \tilde z_1)$ with entries from $\{1, 2, \ldots, m + 1\}$~\cite{PitmanStanley}. 
This can be seen by setting $x_{i, j} - \tilde z_i$ to be the number of entries from $[j]$ in the $i$-th row of the skew-shape. 
Note that for any $n \geq 1$, we have $n\PS_k^{m}(\bm y, \bm z) = \PS_k^{m}(n \bm y, n \bm z)$ and hence, we have a similar plane partition interpretation for lattice points in dilates of the polytope as well.

These polytopes were defined by Pitman-Stanley \cite[Section 5]{PitmanStanley} and studied by Dugan-Hegarty-Morales-Raymond \cite{mgeneralized}. 
The polytope $\PS_k^{m}(\bm y, \bm z)$ can be viewed as a marked order polytope where the underlying poset is a product of two chains, with $m+2$ and $k$ elements, respectively.

\begin{example}
    For $\bm y = (2, 2, 0, 3, 0)$ and $\bm z = (0, 1, 1, 2, 1)$, the marked poset associated to $\PS_5^3(\bm y, \bm z)$ is shown in \Cref{fig:PSex}.
\end{example}

\begin{figure}[H]
    \centering
    \begin{tikzpicture}[scale = 0.75, xscale = 1.4]
        \node (00) at (0, 0) {\color{red}$\tilde z_1$};
        \node (01) at (0, 1) {\color{red}$\tilde z_2$};
        \node (02) at (0, 2) {\color{red}$\tilde z_3$};
        \node (03) at (0, 3) {\color{red}$\tilde z_4$};
        \node (04) at (0, 4) {\color{red}$\tilde z_5$};

        \node (10) at (1, 0 + 0.5) {$x_{1,1}$};
        \node (11) at (1, 1 + 0.5) {$x_{2,1}$};
        \node (12) at (1, 2 + 0.5) {$x_{3,1}$};
        \node (13) at (1, 3 + 0.5) {$x_{4,1}$};
        \node (14) at (1, 4 + 0.5) {$x_{5,1}$};

        \node (20) at (2, 0 + 1) {$x_{1,2}$};
        \node (21) at (2, 1 + 1) {$x_{2,2}$};
        \node (22) at (2, 2 + 1) {$x_{3,2}$};
        \node (23) at (2, 3 + 1) {$x_{4,2}$};
        \node (24) at (2, 4 + 1) {$x_{5,2}$};

        \node (30) at (3, 0 + 1.5) {$x_{1,3}$};
        \node (31) at (3, 1 + 1.5) {$x_{2,3}$};
        \node (32) at (3, 2 + 1.5) {$x_{3,3}$};
        \node (33) at (3, 3 + 1.5) {$x_{4,3}$};
        \node (34) at (3, 4 + 1.5) {$x_{5,3}$};

        \node (40) at (4, 2) {\color{red}$\tilde y_1$};
        \node (41) at (4, 3) {\color{red}$\tilde y_2$};
        \node (42) at (4, 4) {\color{red}$\tilde y_3$};
        \node (43) at (4, 5) {\color{red}$\tilde y_4$};
        \node (44) at (4, 6) {\color{red}$\tilde y_5$};

        \node at (5.5, 3) {$\rightarrow$};
    \end{tikzpicture}
    \hspace{0.4cm}
    \begin{tikzpicture}[scale = 0.75]
        \foreach \x in {1,2,3}{
        \foreach \y in {0,...,4}{
        \node[circle, fill = black, inner sep = 2pt] (\x\y) at (\x, \y + 0.5*\x) {};
        }
        }
        \node[circle, fill = red, inner sep = 2pt, label = left:{\textcolor{red}{0}}] (00) at (0, 0) {};
        \node[circle, fill = red, inner sep = 2pt, label = left:{\textcolor{red}{1}}] (01) at (0, 1) {};
        \node[circle, fill = red, inner sep = 2pt, label = left:{\textcolor{red}{2}}] (02) at (0, 2) {};
        \node[circle, fill = red, inner sep = 2pt, label = left:{\textcolor{red}{4}}] (03) at (0, 3) {};
        \node[circle, fill = red, inner sep = 2pt, label = left:{\textcolor{red}{5}}] (04) at (0, 4) {};

        \node[circle, fill = red, inner sep = 2pt, label = right:{\textcolor{red}{2}}] (40) at (4, 2) {};
        \node[circle, fill = red, inner sep = 2pt, label = right:{\textcolor{red}{4}}] (41) at (4, 3) {};
        \node[circle, fill = red, inner sep = 2pt, label = right:{\textcolor{red}{4}}] (42) at (4, 4) {};
        \node[circle, fill = red, inner sep = 2pt, label = right:{\textcolor{red}{7}}] (43) at (4, 5) {};
        \node[circle, fill = red, inner sep = 2pt, label = right:{\textcolor{red}{7}}] (44) at (4, 6) {};
        
        \foreach \x in {0,...,4}{
        \draw[thin] (\x0) -- (\x1) -- (\x2) -- (\x3) -- (\x4);
        }
        \foreach \y in {0,...,4}{
        \draw[thin] (0\y) -- (1\y) -- (2\y) -- (3\y) -- (4\y);
        }
    \end{tikzpicture}
    \caption{Marked poset associated to a generalized Pitman-Stanley polytope.}
    \label{fig:PSex}
\end{figure}
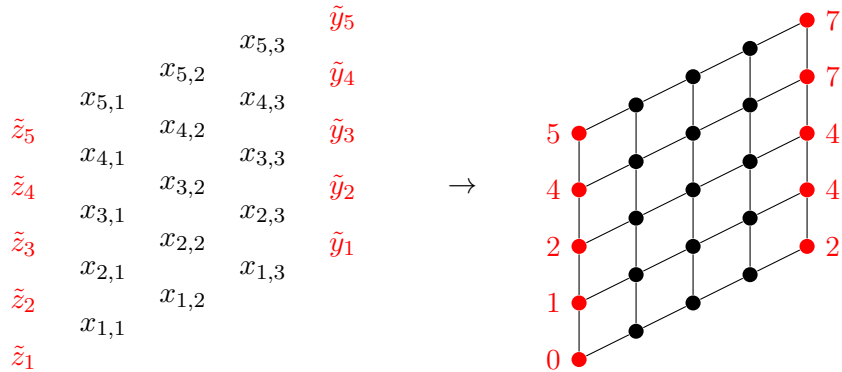

Setting $\mathcal{F}$ to be the family of cell posets associated to skew-shapes, we see that any product of two chains is in $\mathcal{F}$ and, by \cite[Proposition 4.1]{skewshapes}, this family also satisfies the conditions in \Cref{thm:posetfamilies}. 
This gives us the following.

\begin{thm}\label{thm:mgeneralizedgeneral}
    For any $\bm y, \bm z \in \Z_{\geq 0}^k$ and $m \geq 1$, the $m$-generalized Pitman-Stanley polytope $\PS_k^{m}(\bm y, \bm z)$ is Ehrhart positive.
\end{thm}

When all terms of $\bm z$ are $0$, lattice points correspond to plane partitions of the straight shape $(\tilde y_k, \ldots, \tilde y_1)$. 
For any $\bm y \in \Z_{\geq 0}^k$, we set $\PS_k^m(\bm y) := \PS_k^m(\bm y, \bm 0)$. Since $\tilde z_1 \leq \cdots \leq \tilde z_k\leq \tilde y_1\leq \cdots \leq \tilde y_k$, another consequence of \Cref{thm:posetfamilies} is the following.

\begin{thm}\label{thm:conjmgeneralized}
    For any $\bm y \in \Z_{\geq 0}^k$, we have that $|\PS_k^m(\bm y) \cap \Z^{km}|$ is a polynomial in $\mathbb{Q}[\bm y]$ with nonnegative coefficients.
\end{thm}

Thereby we answer \cite[Conjecture 3.4]{generalizedPitmanStanleyAbstract} due to Dugan-Hegarty-Morales-Raymond in the affirmative. Theorem~\ref{thm:mgeneralizedgeneral} was mentioned as a conjecture in~\cite[Conjecture 9.3]{skewshapes} attributed to an upcoming article by Dugan-Hegarty-Morales-Raymond.

\subsection{Gelfand-Tsetlin polytopes}

Let $\bm y, \bm z \in \Z_{\geq 0}^k$ and $m \geq 1$. 
For any $i \in [k]$, set $\tilde y_i = y_1 + \cdots + y_i$ and similarly define $\tilde z_i$. 
The \emph{skew Gelfand-Tsetlin polytope}, denoted $\OGT_k^m(\bm y, \bm z)$, is the set of all points $(x_{i, j}) \in \R^{km}$ that satisfy
\begin{equation*}
    x_{i, j} \leq x_{i, j + 1} \leq x_{i + 1, j}
\end{equation*}
where we set $x_{i, 0} = \tilde z_i$ and $x_{i, m + 1} = \tilde y_i$ for all $i \in [k]$. 
The lattice points of this polytope are in correspondence with semi-standard Young tableaux of skew shape $(\tilde y_k, \ldots, \tilde y_2, \tilde y_1)/ (\tilde z_k, \ldots, \tilde z_2, \tilde z_1)$ with entries from $\{1, 2, \ldots, m + 1\}$. 
This can be seen by setting $x_{i, j} - \tilde z_i$ to be the number of entries from $[j]$ in the $i$-th row of the skew-shape. 
Note that for any $n \geq 1$, we have $n\OGT_k^m(\bm y, \bm z) = \OGT_k^m(n\bm y, n\bm z)$.

Gelfand-Tsetlin patterns (for straight shapes) were introduced in \cite{GelfandTsetlinPolytopes}, and the corresponding polytopes have also been studied as marked order polytopes \cite{ArdilaBliemSalazar}. 
These polytopes can be viewed as marked order polytopes as follows. 
Given $\bm y, \bm z \in \Z_{\geq 0}^k$ and $m \geq 1$, construct the poset $P$ on $\{(i, j) \colon i \in [k], j \in [0, m + 1]\}$ with cover relations
\begin{equation*}
    (i, j) \prec_P (i, j + 1) \prec_P (i + 1, j).
\end{equation*}
Let $A = \{(i, j) \in P \colon j \in \{0, m + 1\}\}$ and set $\l : A \rightarrow \Z$ to be the map given by $\l(i, 0) = \tilde z_i$ and $\l(i, m + 1) = \tilde y_i$ for all $i \in [k]$. 
We then have $\mathcal{O}_{P,A}(\l) = \OGT_k^m(\bm y, \bm z)$.

\begin{example}
    For $\bm y = (1, 0, 1, 2)$ and $\bm z = (0, 0, 1, 0)$, the marked poset associated to $\OGT_4^2(\bm y, \bm z)$ is shown in \Cref{fig:GTex}.
\end{example}

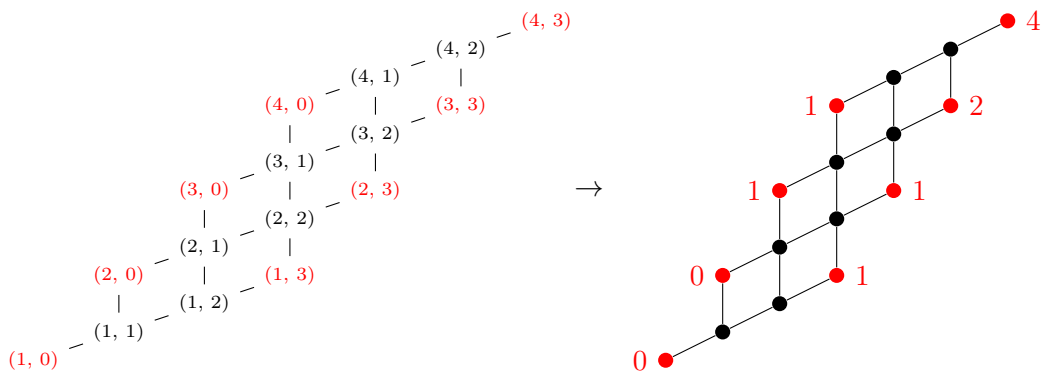
\begin{figure}[h]
    \centering
    \begin{tikzpicture}[scale = 0.75, xscale = 1.5]
        \node (00) at (0, 0) {\color{red} \tiny (1, 0)};
        \node (10) at (1, 0.5) {\tiny (1, 1)};
        \node (20) at (2, 1) {\tiny (1, 2)};
        \node (30) at (3, 1.5) {\color{red} \tiny (1, 3)};

        \node (01) at (1 + 0, 1.5 + 0) {\color{red} \tiny (2, 0)};
        \node (11) at (1 + 1, 1.5 + 0.5) {\tiny (2, 1)};
        \node (21) at (1 + 2, 1.5 + 1) {\tiny (2, 2)};
        \node (31) at (1 + 3, 1.5 + 1.5) {\color{red} \tiny (2, 3)};

        \node (02) at (2 + 0, 3 + 0) {\color{red} \tiny (3, 0)};
        \node (12) at (2 + 1, 3 + 0.5) {\tiny (3, 1)};
        \node (22) at (2 + 2, 3 + 1) {\tiny (3, 2)};
        \node (32) at (2 + 3, 3 + 1.5) {\color{red} \tiny (3, 3)};

        \node (03) at (3 + 0, 4.5 + 0) {\color{red} \tiny (4, 0)};
        \node (13) at (3 + 1, 4.5 + 0.5) {\tiny (4, 1)};
        \node (23) at (3 + 2, 4.5 + 1) {\tiny (4, 2)};
        \node (33) at (3 + 3, 4.5 + 1.5) {\color{red} \tiny (4, 3)};

        \foreach \y in {0,...,3}{
        \draw[thin] (0\y) -- (1\y) -- (2\y) -- (3\y);
        }
        \draw[thin] (10) -- (01);
        \draw[thin] (20) -- (11);
        \draw[thin] (30) -- (21);
        
        \draw[thin] (11) -- (02);
        \draw[thin] (21) -- (12);
        \draw[thin] (31) -- (22);

        \draw[thin] (12) -- (03);
        \draw[thin] (22) -- (13);
        \draw[thin] (32) -- (23);

        \node at (6.5, 3) {$\rightarrow$};
    \end{tikzpicture}
    \begin{tikzpicture}[scale = 0.75]
        \node[circle, fill = red, inner sep = 2pt, label = left:{\textcolor{red}{0}}] (00) at (0, 0) {};
        \node[circle, fill = black, inner sep = 2pt] (10) at (1, 0.5) {};
        \node[circle, fill = black, inner sep = 2pt] (20) at (2, 1) {};
        \node[circle, fill = red, inner sep = 2pt, label = right:{\textcolor{red}{1}}] (30) at (3, 1.5) {};

        \node[circle, fill = red, inner sep = 2pt, label = left:{\textcolor{red}{0}}] (01) at (1 + 0, 1.5 + 0) {};
        \node[circle, fill = black, inner sep = 2pt] (11) at (1 + 1, 1.5 + 0.5) {};
        \node[circle, fill = black, inner sep = 2pt] (21) at (1 + 2, 1.5 + 1) {};
        \node[circle, fill = red, inner sep = 2pt, label = right:{\textcolor{red}{1}}] (31) at (1 + 3, 1.5 + 1.5) {};

        \node[circle, fill = red, inner sep = 2pt, label = left:{\textcolor{red}{1}}] (02) at (2 + 0, 3 + 0) {};
        \node[circle, fill = black, inner sep = 2pt] (12) at (2 + 1, 3 + 0.5) {};
        \node[circle, fill = black, inner sep = 2pt] (22) at (2 + 2, 3 + 1) {};
        \node[circle, fill = red, inner sep = 2pt, label = right:{\textcolor{red}{2}}] (32) at (2 + 3, 3 + 1.5) {};

        \node[circle, fill = red, inner sep = 2pt, label = left:{\textcolor{red}{1}}] (03) at (3 + 0, 4.5 + 0) {};
        \node[circle, fill = black, inner sep = 2pt] (13) at (3 + 1, 4.5 + 0.5) {};
        \node[circle, fill = black, inner sep = 2pt] (23) at (3 + 2, 4.5 + 1) {};
        \node[circle, fill = red, inner sep = 2pt, label = right:{\textcolor{red}{4}}] (33) at (3 + 3, 4.5 + 1.5) {};

        \foreach \y in {0,...,3}{
        \draw[thin] (0\y) -- (1\y) -- (2\y) -- (3\y);
        }
        \draw[thin] (10) -- (01);
        \draw[thin] (20) -- (11);
        \draw[thin] (30) -- (21);
        
        \draw[thin] (11) -- (02);
        \draw[thin] (21) -- (12);
        \draw[thin] (31) -- (22);

        \draw[thin] (12) -- (03);
        \draw[thin] (22) -- (13);
        \draw[thin] (32) -- (23);
    \end{tikzpicture}
    \caption{Marked poset associated to a skew Gelfand-Tsetlin polytope.}
    \label{fig:GTex}
\end{figure}

We now show Ehrhart positivity for these polytopes using \Cref{prop:productformula}, thereby proving a conjecture of Alexandersson-Alhajjar \cite[Conjecture 7]{GelfandTsetlinconjecture}.

\begin{thm}\label{thm:conjGelfand-Tsetlin}
    For any $\bm y, \bm z \in \Z_{\geq 0}^k$ and $m \geq 1$, the skew Gelfand-Tsetlin polytope $\OGT_k^m(\bm y, \bm z)$ is Ehrhart positive.
\end{thm}

\begin{proof}
    Let the posets $A \subseteq P$ be as constructed above. Consider $A=\{a_r=(i_r,j_r)\colon r\in [2k]\}$ naturally labelled. 
    Note that for any $r \in [2k - 1]$, we must have
    \begin{itemize}
        \item $j_r = 0$, $j_{r + 1} = m + 1$, and $i_r \geq i_{r + 1}$, or
        \item $j_r = m + 1$, $j_{r + 1} = 0$, and $i_r < i_{r + 1}$, or
        \item $j_r = j_{r + 1}$ and $i_r + 1 = i_{r + 1}$.
    \end{itemize}
    Let $r \in [2k]$ and $I \subset I'$ be ideals of $P$ such that $a_1, \ldots, a_r \in I$, $a_{r + 1} \in I'\setminus I$, and $a_{r + 2}, \ldots, a_{2k} \notin I'$. 
    The expression from \Cref{prop:productformula} involves order polynomials of posets of the form $J = I' \setminus(I \mathop{\cup} \uparrow a_{r + 1})$. 
    The result follow since these posets are always cell posets of skew-shapes:

    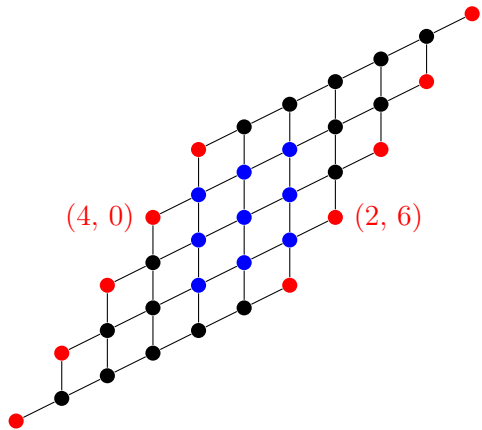
\begin{figure}[h]
        \centering
        \begin{tikzpicture}[scale = 0.6]
            \node[circle, fill = red, inner sep = 2pt] (00) at (0, 0) {};
            \foreach \x in {1,...,5}{
            \node[circle, fill = black, inner sep = 2pt] (\x0) at (\x, 0.5*\x) {};
            }
            \node[circle, fill = red, inner sep = 2pt] (60) at (6, 3) {};

            \node[circle, fill = red, inner sep = 2pt] (01) at (1, 1.5) {};
            \foreach \x in {1,2}{
            \node[circle, fill = black, inner sep = 2pt] (\x1) at (1 + \x, 1.5 + 0.5*\x) {};
            }
            \foreach \x in {3,4,5}{
            \node[circle, fill = blue, inner sep = 2pt] (\x1) at (1 + \x, 1.5 + 0.5*\x) {};
            }
            \node[circle, fill = red, inner sep = 2pt, label = right:{\textcolor{red}{(2, 6)}}] (61) at (7, 4.5) {};

            \node[circle, fill = red, inner sep = 2pt] (02) at (2, 3) {};
            \foreach \x in {1,5}{
            \node[circle, fill = black, inner sep = 2pt] (\x2) at (2 + \x, 3 + 0.5*\x) {};
            }
            \foreach \x in {2,3,4}{
            \node[circle, fill = blue, inner sep = 2pt] (\x2) at (2 + \x, 3 + 0.5*\x) {};
            }
            \node[circle, fill = red, inner sep = 2pt] (62) at (8, 6) {};

            \node[circle, fill = red, inner sep = 2pt, label = left:{\textcolor{red}{(4, 0)}}] (03) at (3, 4.5) {};
            \foreach \x in {4,5}{
            \node[circle, fill = black, inner sep = 2pt] (\x3) at (3 + \x, 4.5 + 0.5*\x) {};
            }
            \foreach \x in {1,2,3}{
            \node[circle, fill = blue, inner sep = 2pt] (\x3) at (3 + \x, 4.5 + 0.5*\x) {};
            }
            \node[circle, fill = red, inner sep = 2pt] (63) at (9, 7.5) {};

            \node[circle, fill = red, inner sep = 2pt] (04) at (4, 6) {};
            \foreach \x in {1,...,5}{
            \node[circle, fill = black, inner sep = 2pt] (\x4) at (4 + \x, 6 + 0.5*\x) {};
            }
            \node[circle, fill = red, inner sep = 2pt] (64) at (10, 9) {};

            \foreach \y in {0,...,4}{
            \draw[thin] (0\y) -- (1\y) -- (2\y) -- (3\y) -- (4\y) -- (5\y) -- (6\y);
            }

            \draw[thin] (10) -- (01);
            \draw[thin] (20) -- (11);
            \draw[thin] (30) -- (21);
            \draw[thin] (40) -- (31);
            \draw[thin] (50) -- (41);
            \draw[thin] (60) -- (51);

            \draw[thin] (11) -- (02);
            \draw[thin] (21) -- (12);
            \draw[thin] (31) -- (22);
            \draw[thin] (41) -- (32);
            \draw[thin] (51) -- (42);
            \draw[thin] (61) -- (52);

            \draw[thin] (12) -- (03);
            \draw[thin] (22) -- (13);
            \draw[thin] (32) -- (23);
            \draw[thin] (42) -- (33);
            \draw[thin] (52) -- (43);
            \draw[thin] (62) -- (53);

            \draw[thin] (13) -- (04);
            \draw[thin] (23) -- (14);
            \draw[thin] (33) -- (24);
            \draw[thin] (43) -- (34);
            \draw[thin] (53) -- (44);
            \draw[thin] (63) -- (54);
        \end{tikzpicture}
        \caption{Subposet $S$ (in blue) containing $J$ if $a_r = (4, 0)$ and $a_{r + 1} = (2, 6)$.}
        \label{fig:GTrec}
    \end{figure}
    
    When $j_r = 0$ and $j_{r + 1} = m + 1$, we consider the subposet (see \Cref{fig:GTrec})
    \begin{equation*}
        S = \{(i, j) \in P \colon i \in [i_{r + 1}, i_r] \text{ and }i_r + j_r < i + j < i_{r + 1} + j_{r + 1}\}.
    \end{equation*}
    We can obtain $J$ from $S$ by removing ideals and filters. 
    Since $S$ is always a cell poset of a skew-shape (product of two chains), so is $J$. 
    The case $j_r = m + 1$ and $j_{r + 1} = 0$ is similar.

    If $j_r = j_{r + 1} = 0$, then consider the subposet
    \begin{equation*}
        S = \{(i, j) \in P \colon i \in [i', i_r] \text{ and } i_r < i + j < i' + m + 1\}
    \end{equation*}
    where $i' \in [k]$ is the smallest such that $(i', m + 1)$ appears after $a_r$ in the total order on $A$. 
    We can obtain $J$ from $S$ by removing ideals and filters, and since $S$ is again a product of two chains, the result follows. 
    The case $j_r = j_{r + 1} = m + 1$ is analogous.
\end{proof}

We conclude by mentioning that similar techniques can be used to show that related polytopes studied in the literature also satisfy multivariate Ehrhart positivity.

For one, for certain faces of Gelfand-Tsetlin polytopes $\OGT_k^m(\bm y, \bm z)$, the lattice points correspond to \emph{flagged} semi-standard Young tableaux, introduced by Lascoux-Sch\"utzenberger~\cite{LS}: given weakly increasing vectors $\bm a, \bm b \in \Z^k_{\geq 0}$ such that $0 \leq a_i < b_i \leq m + 1$, the lattice points in the face of $\OGT_k^m(\bm y, \bm z)$ given by the equalities $x_{i, j} = x_{i, j - 1}$ for all  $j \notin [a_i + 1, b_i]$ and all $i \in [k]$
correspond to semi-standard Young tableaux of skew-shape $(\tilde y_k, \ldots, \tilde y_2, \tilde y_1)/ (\tilde z_k, \ldots, \tilde z_2, \tilde z_1)$ such that all entries in the $i$-th row are from $[a_i + 1, b_i]$. This face can be viewed as a marked order polytope itself, namely on the poset obtained by identifying the elements involved in the defining equalities. The same method as in Theorem~\ref{thm:conjGelfand-Tsetlin} can then be applied to show Ehrhart positivity. These faces of Gelfand-Tsetlin polytopes were studied by Kogan \cite{Kogan}. This also answers a special case of \cite[Conjecture 10]{GelfandTsetlinconjecture} since key polynomials associated to $312$-avoiding permutations are flagged Schur polynomials \cite[Theorem 14.1]{PSChains}.

Finally, similar techniques can also be used to show that a type $C$ analogue of Gelfand-Tsetlin polytopes, introduced by Berenstein-Zelevinsky~\cite{typec} and viewed as a marked order polytope by Ardila-Bliem-Salazar~\cite[Section 4.2]{ArdilaBliemSalazar}, satisfies multivariate Ehrhart positivity.

\textbf{Acknowledgements:} We thank Per Alexandersson for pointing us to further applications, and  Petter Brändén and Benjamin Schröter for helpful comments.

\bibliographystyle{siam}
\bibliography{bibliography}

\end{document}